\documentclass[12pt]{article}
\usepackage[T2A]{fontenc}
\usepackage{amsmath,amstext}
\usepackage{amsfonts}
\usepackage{amsthm}
\usepackage{amssymb}
\usepackage{srcltx}
\usepackage[unicode]{hyperref}

\renewcommand{\phi}{\varphi}
\renewcommand{\le}{\leqslant}
\renewcommand{\ge}{\geqslant}
\renewcommand{\a}{\alpha}

\newcommand{\mC}{\mathcal{C}}

\newcommand{\Leb}{\mathop{\mathrm{Leb}}}
\newcommand{\Lip}{\mathop{\mathrm{Lip}}}
\newcommand{\dist}{\mathop{\mathrm{dist}}}
\newcommand{\bbN}{\mathbb{N}}
\newcommand{\const}{\mathop{\mathrm{const}}}

\newtheorem{theorem}{Theorem}

\newtheorem{corollary}{Corollary}
\newtheorem{lemma}{Lemma}

\theoremstyle{definition}
\newtheorem{rem}{Remark}
\newtheorem*{rem*}{Remark}
\newtheorem{definition}{Definition}
\newtheorem{definition*}{Definition}

\title{Special Ergodic Theorem for hyperbolic maps.}

\author{Victor Kleptsyn\footnote{CNRS, Institute of Mathematical Research of Rennes (IRMAR, UMR~6625 du CNRS); victor.kleptsyn@univ-rennes1.fr}\addtocounter{footnote}{2}\ \ and Dmitry Ryzhov\footnote{Chebyshev Laboratory (Department of Mathematics and Mechanics, St.-Petersburg State University, Russia)}}

\begin{document}
\maketitle
\begin{abstract}
Let $f\colon M\to M$ be a smooth transformation of a compact manifold~$M$, admitting an observable SRB measure~$\mu$. For a continuous test function $\varphi\colon M\to \mathbb R$ and a constant $\alpha>0$, consider the set $X_{\varphi,\alpha}$ of the initial points for which the Birkhoff time averages of the function~$\varphi$ differ from its $\mu$--space average by at least~$\alpha$. As the measure~$\mu$ is an SRB one, the intersection of $X_{\varphi,\alpha}$ with the basin of attraction of~$\mu$ should have zero Lebesgue measure. 
\par
The \emph{special ergodic theorem}, whenever it holds, claims that, moreover, this intersection has the Hausdorff dimension less than the dimension of~$M$. We prove that for $C^1$-local diffeomorphisms, the special ergodic theorem follows from the dynamical large deviations principle. Applying a theorem of L.~S.~Young, we conclude that in the case of $C^2$-diffeomorphisms, the special ergodic theorem holds for hyperbolic attractors as well as for Anosov maps.
\end{abstract}

\section*{Introduction}

Let $f:M\to M$ be a smooth map of a smooth compact manifold $M$ to itself. The famous Birkhoff ergodic theorem states that if $\mu$ is an ergodic invariant measure for $f$, then for any continuous 
function $\varphi\in C(M)$ and for $\mu$-almost every $x\in M$, at the point $x$ the time averages $\varphi_n$ of the function $\varphi$ converge to its space average~$\bar{\varphi}$:
\begin{equation}
\lim\limits_{n\to\infty} \phi_n(x)= \bar\varphi, \quad \text{ where } \quad \phi_n(x):=\frac{1}{n}\sum_{k=0}^{n-1} \phi\circ f^k(x), \quad \bar\phi:=\int_X \phi\,d\mu.
\end{equation}

The conclusion of this theorem has one large drawback. It is usually much more natural (in particular, from the ``physical'' point of view) to take an initial point $x$, which is generic in the sense of the Lebesgue measure, and if the measure $\mu$ is supported on some ``thin'' set (for instance, on a point), one cannot deduce any conclusions for initial points outside this set.

This leads to another famous notion, the one of Sinai-Ruelle-Bowen measure (or SRB measure for short):
\begin{definition}
A measure $\mu$ is an \emph{(observable) SRB measure} for $f:M\to M$ with a basin of attraction $U\subset M$ if for Lebesgue-almost every $x\in U$ one has 
\begin{equation}\label{eq:time-averages}
\varphi_n (x) \to \bar{\varphi} := \int_M \varphi \, d\mu.
\end{equation}
\end{definition}

Not all the systems possess SRB measures. The first example of a system with no SRB measures, the \emph{heteroclinic attractor}, was given by Bowen himself. However, under certain assumptions, e.g. hyperbolicity of the system, one can guarantee its existence; we refer the reader to~\cite{Young-SRB} for the survey of this subject, as well as to~\cite[Sec. 14]{KH06}.

But even the conclusion in the definition of an SRB measure sometimes turns out to be insufficient. In some situations~(cf. \cite{IKS},~\cite{IN},~\cite{KS}), it is necessary not only to conclude that the set of the ``bad'' initial conditions has zero Lebesgue measure, but also to state the same for its images under H\"older (but not necessarily absolutely continuous!) conjugacy between two dynamical systems.

One of the ways of handling this difficulty bases on the analysis of Hausdorff dimension of this set. It leads to the concept of \emph{special ergodic theorem}, proposed by Yu. Ilyashenko, to which this paper is devoted. We will now introduce some auxiliary notations, and present this concept in Def.~\ref{def:spec-erg}. 

For any $\varphi\in C(M)$, $\alpha>0$, define the set
\begin{equation*}
	X_{\a,\phi}:=\left\{ x\in X\colon \varlimsup\limits_{n\to\infty} |\phi_n(x)-\bar\phi|>\a \right\}.
\end{equation*}
In other words, $X_{\alpha,\varphi}$ is the set of points for which the statement ``time averages converge to the space one'' is violated ``in an essential way''.

Let now $f:M\to M$ be a $C^1$-local diffeomorphism, and $U\subset M$ a \emph{trapping domain}, $f(\overline{U}) \subset U$. Assume that $\mu$ is an SRB-measure supported in~$U$ and hence in the maximal attractor $A:=\cap_n f^n(U)$. 

\begin{definition}\label{def:spec-erg}
Say that for $(f,\mu,U)$ holds the \emph{special ergodic theorem}, if for any continuous function $\varphi\in C(M)$ and any $\alpha>0$ the Hausdorff dimension of the set $X_{\a,\phi}\cap U$ is strictly less than the dimension of the phase space:
\begin{equation}\label{eq:set}
	\forall \varphi\in C(M), \alpha>0 \qquad \dim_H (X_{\a,\phi}\cap U)<\dim M.
\end{equation}
\end{definition}

\begin{rem}
Definition~\ref{def:spec-erg}, as well as the definitions and statements below, can be applied in case $U=M$: we do not require the strict contraction~$U\neq f(U)$.
\end{rem}

\begin{rem}
Formally speaking, to state the property \eqref{eq:set} one doesn't need to claim the existence of SRB measure. Though, the conclusion of the definition immediately implies that $\mu$ is an SRB measure. Indeed, the set~$X_{0,\phi}\cap U$ of points in~$U$ for which the equality~\eqref{eq:time-averages} between time and space averages does not hold is exhausted by a countable number of sets $X_{1/n,\varphi}\cap U$ of zero Lebesgue measure.
\end{rem}

Having introduced this notion, Yu.~Ilyashenko~\cite{IKS} proved that the special ergodic theorem holds for the circle doubling map; later, P.~Saltykov~\cite{Saltykov} proved it also for the linear Anosov map on the 2-torus. Analogous questions were also considered in the work of Gurevich and Tempelman (cf.~\cite{GT02}).

The aim of this note is to prove the special ergodic theorem for a larger class of maps~--- in particular, for all the hyperbolic ones. Namely, we notice and prove, that the special ergodic theorem can be obtained as a corollary of another property of the system, \emph{the dynamical large deviations principle}:
\begin{definition}\label{def:DLDP}
Let a $C^1$-local diffeomorphism $f:M\to M$, a trapping domain $U\subset M$ and an invariant measure $\mu$ supported in $U$ be given. Say, that for $(f,\mu,U)$ holds the \emph{dynamical large deviations principle} if for any $\varphi\in C(M)$, $\alpha>0$ there exist constants\footnote{For a fixed function~$\phi$, the function~$h=h(\a)$ is called the ''rate function'' or ''Cram\'er function''; see, for example, \cite{Bellet-Young,Varadhan}.
} 
 $C=C(\a,\phi),h=h(\a,\phi)>0$
such that
\begin{equation}\label{eq:ld}
\forall n\in\bbN \quad \Leb(X_{\a,\phi,n}\cap U)\le C e^{-nh},
\end{equation}
where
$$   
X_{\a,\phi,n} :=\{ x\in M \colon |\phi_n(x)-\bar{\phi}|\ge \a\}.
$$
\end{definition}

\begin{rem}
The measure $\mu$ in Definition~\ref{def:DLDP} is automatically an SRB measure (attracting Lebesgue-a.e. point from~$U$) due to the Borel--Cantelli Lemma type arguments.
\end{rem}

\begin{theorem}[Main result]\label{t:main}
If for a $C^1$-local diffeomorphism $f:M\to M$, a trapping domain~$U$ and an invariant measure~$\mu$ holds the dynamical large deviations principle, then for this system also holds the special ergodic theorem.
\end{theorem}

As the reader will see in the next section, the proof of Theorem~\ref{t:main} is indeed rather simple (moreover, one needs only Lipschitz condition on the map $f$). However, the conclusion itself seems very interesting to us~--- especially, due to the possibilities to apply it to the study of the skew products.


The dynamical large deviations principle is a dynamical counterpart of the large deviations principle of the probability theory. The latter states that under some assumptions on the distribution of i.i.d. random variables $\xi_j$, for any $\a>0$ the probability of a large deviation $P(|(\xi_1+\dots+\xi_n)/n - E \xi|>\a)$ decreases exponentially in~$n$ as it goes to infinity.

The dynamical large deviations principle is well-studied, and was proved for all hyperbolic
attractors, that is, maximal attractors that are hyperbolic sets, by L.-S.~Young~\cite{Young}, see Thm.~\ref{you}. (In particular, this theorem also covers the case of Anosov diffeomorphisms.)
There are also large deviations results for non-uniformly hyperbolic systems, see~\cite{Melbourne-Nicol} and~\cite{Bellet-Young}, but in these works the property~\eqref{eq:ld} is proven for the invariant measure~$\mu$ instead of the Lebesgue measure. Thus, in the latter case, one can apply Theorem~\ref{t:main} only if the measure $\mu$ is absolutely continuous with density bounded away from zero.
\begin{theorem}[L.-S. Young, \cite{Young}, Thm.~2(2)]\label{you}
If $f:M\to M$ is a $C^2$-diffeomorphism, $U$ is a trapping domain, and the maximal attractor $A:=\cap_n f^n(U)$ is hyperbolic, then on $A$ is supported an SRB measure $\mu$ such that the dynamical large deviations principle holds for $(f,\mu,U)$.
\end{theorem}

Thus, ``plugging'' this theorem into our main result, we immediately obtain
\begin{corollary}
The special ergodic theorem holds for all hyperbolic attractors of $C^2$-diffeomorphisms, in particular, for all $C^2$~Anosov diffeomorphisms.
\end{corollary}

\section*{Proof of the main theorem}
Let $f\colon M\to M$ be a $C^1$-local diffeomorphism on a $d$-dimensional compact manifold~$M$ with a trapping region $U$ and SRB measure $\mu$ supported on it, for which the dynamical large deviations principle holds. Choose and fix $\varphi\in C(M)$, $\a>0$, and denote $L:=\max(\Lip(f),2)$. Note, that as $f$ is a Lipschitz map, and $X_{\alpha,\phi}$ is $f$-invariant, it suffices to estimate $\dim_H ( X_{\alpha,\phi}\cap f(\overline U))$ instead of~$\dim_H ( X_{\alpha,\phi}\cap U)$. 

We are going to prove that 
$$
\dim_H \left(X_{\a,\phi}\cap f(\overline U)\right) \le d_0:=d-\frac{h(\a/2,\varphi)}{\ln L}< d.
$$ 
That is, for each $d'>d_0$ we want to find a cover of $X_{\a,\phi}\cap f(\overline U)$ with arbitrarily small $d'$-dimensional volume.

To construct such a cover, we first choose (and fix) $\delta>0$ such that
$$
\forall x,y\in X, \, \dist(x,y)<\delta \quad |\varphi(x)-\varphi(y)|<\alpha/2.
$$

Then we have the following lemma:
\begin{lemma}\label{l:ball}
If $x\in X_{\a,\phi,n}$, then $B(x,r)\subset X_{\a/2,\phi,n}$, where $r=r(n):=L^{-n}\delta$ and $B(x,r)$ is a ball centered at $x$ with radius $r$.
\end{lemma}
\begin{proof}
Let $y\in B(x,r)$. Then $\dist(x,y)< L^{-n}\delta$ and consequently
\begin{equation*}
	\dist(f^k(x),f^k(y))<\delta\text{ for } k=0,1,\dots n.
\end{equation*}
Therefore, by the choice of $\delta$, we have
\begin{equation*}
	|\phi(f^k(x))-\phi(f^k(y))|<\a/2\text{ for } k=0,1,\dots n,
\end{equation*}
what implies (by averaging over~$k$) that $|\phi_n(x)-\phi_n(y)|<\a/2$. Therefore
\begin{equation*}
	|\phi_n(y)-\bar\phi| \ge |\bar\phi-\phi_n(x)|-|\phi_n(x)-\phi_n(y)| \ge \a-\a/2 = \a/2.
\end{equation*}
\end{proof}

Now, for any $n$, choose a maximal $\frac{r(n)}{2}$-separated set $S_n$ in~$f(\overline U)$. The union of $\frac{r(n)}{2}$-neighborhoods of elements of $S_n$ is a cover of $f(\overline U)$, and in particular, of~$X_{\a,\phi,n} \cap f(\overline U)$. Remove from this cover the neighborhoods that do not intersect $X_{\a,\varphi,n}\cap f(\overline U)$: let 
$$
S_n':=\left\{x\in S_n \colon 
B_{r(n)/2}(x)\cap (X_{\a,\phi,n}\cap f(\overline U))\neq \emptyset \right\}
$$ and 
$$
\mC_n := \left\{B_{r(n)/2} (x) \mid x\in S_n' \right\}.
$$

By construction, $\mC_n$ is still a cover of $X_{\a,\phi,n}\cap f(\overline U)$. Also, noticing that any point of $X_{\a,\phi}\cap f(\overline U)$ belongs to an infinite number of $X_{\a,\phi,n}\cap f(\overline U)$, we see that for any $N\in\mathbb N$ the union of $\mC_n$ for $n\ge N$ is a cover of $X_{\a,\phi}$.

Let us estimate the $d'$-dimensional volumes of these covers. First, estimate the cardinality of $\mC_n$: 
note that radius $r(n)/4$ balls centered at the elements of $S_n$ are pairwise disjoint. Hence, the volume of the union
$$
V_n:= \bigcup_{x\in S_n'} B_{r(n)/4}(x)
$$
is at least $card (\mC_n) \cdot c\cdot r(n)^d$, where the constant $c>0$ depends only on the geometry of the manifold $M$. 

On the other hand, for any $n$ such that the $r(n)/4$-neighborhood of $f(\overline U)$ is included in $U$, due to Lemma~\ref{l:ball} one has $V_n\subset X_{\alpha/2,\phi,n} \cap U$. Hence, due to the assumed large deviations principle, $\Leb(V_n)\le Ce^{-h(\alpha/2,\varphi) n}$. Thus, 
$$
card (\mC_n) \cdot c\cdot r(n)^d \le C e^{-h(\a/2,\phi) n}
$$
and hence
$$
card(\mC_n) \le C' L^{nd} e^{-h(\a/2,\phi) n} = C' L^{n d_0}.
$$

Taking any sufficiently big~$N$, we see that the $d'$-dimensional volume of the union of covers $V_n$ for $n\ge N$ can be estimated as 
$$
\sum_{n=N}^{\infty} card(\mC_n) \cdot r(n)^{d'} \le \const \cdot \sum_{n=N}^{\infty} L^{-n(d'-d_0)}
$$
This series converges, as $d'>d_0$ by the choice of $d'$. Thus, choosing $N$ sufficiently large, we obtain covers of $X_{\alpha,\varphi}\cap f(\overline U)$ with arbitrarily small $d'$-dimensional volumes, and hence $\dim_H X_{\alpha,\varphi} \cap f(\overline U) \le d' $ for any $d'>d_0$. This implies that
$$\dim_H X_{\alpha,\varphi} \cap f(\overline U) \le d_0,$$
what concludes the proof of the theorem.

\subsection*{Acknowledgments}
The authors are very grateful for fruitful discussions to Yu.~Ilyashenko, M.~Blank, A.~Bufetov, A.~Gorodetsky, Ya.~Pesin and P.~Saltykov. Also, the authors would like to thank the organizers of the conference ``\emph{Geometric and Algebraic Structures in Mathematics}'' for bringing them together. Both authors were partially supported by RFBR grant 10-01-00739-a and joint RFBR/CNRS grant 10-01-93115-CNRS-a.
 Research of the second author was supported by the Chebyshev Laboratory (Department of Mathematics and Mechanics, St.-Petersburg State University) under the Russian Federation government grant 11.G34.31.0026.

\begin{small}

\end{small}


\begin{thebibliography}{99}


\bibitem{GT02} {\sc B. M. Gurevich}, {\sc A. A. Tempelman}.
Hausdorff dimension of the set of typical points for Gibbs measures.
\emph{Functional Analysis and its Applications}, \textbf{36}:3 (2002), 225--227.

\bibitem{KH06}
{\sc L.~Barreira, Ya.~Pesin.} Smooth Ergodic Theory and Nonuniformly Hyperbolic Dynamics. \emph{In:  Handbook of dynamical systems. Vol.~1B}, 57--263. Editors:
B.~Hasselblatt, A.~Katok. Elsevier, Amsterdam, 2006.

\bibitem{IKS}
{\sc Yu. Ilyashenko}, {\sc V. Klepstyn}, {\sc P.Saltykov}.
Openness of the set of boundary preserving maps of an annulus with intermingled attracting basins.
\textit{Journal of Fixed Point Theory and Applications},  \textbf{3}:2 (2008), 449--463.

\bibitem{IN}
{\sc Yu. Ilyashenko}, {\sc A. Negut}.
Invisible parts of attractors.
\emph{Nonlinearity}, \textbf{23} (2010), 1199--1219.

\bibitem{KS}
{\sc V. Kleptsyn, P. Saltykov}. On $C^2$-stably intermingled attractors in the
classes of boundary-preserving maps, \emph{Transactions of Moscow Mathematical Society}, \textbf{72}:2 (2011), to appear.

\bibitem{Melbourne-Nicol}
{\sc I. Melbourne}, {\sc M. Nicol}.
Large deviations for nonuniformly hyperbolic systems.
\emph{Trans. Amer. Math. Soc.}, \textbf{360} (2008), 6661--6676.

\bibitem{Bellet-Young}
{\sc L. Rey-Bellet}, {\sc L.-S. Young}.
Large deviations in non-uniformly hyperbolic dynamical systems
\emph{Ergod. Th. \& Dynam. Sys.}, \textbf{28} (2008), 587--612.

\bibitem{Saltykov}
{\sc P. Saltykov}.
Special ergodic theorem for Anosov diffeomorphisms on the 2-torus.
\emph{Functional Analysis and its Applications}, \textbf{45}:1 (2011), 69--78.

\bibitem{Varadhan}
{\sc S. R. S. Varadhan}.
Large deviations.
\textit{Annals of Probability}, \textbf{36}:2 (2008), 397--419.

\bibitem{Young}
{\sc L.-S. Young}.
Some large deviation results for dynamical systems.
\textit{Transactions of the AMS}, \textbf{318}:2 (1990), 525--543.

\bibitem{Young-SRB}
{\sc L.-S. Young}.
What are SRB measures, and which dynamical systems have them?
\textit{J. Statist. Phys.}, \textbf{108} (2002), 733--754.

\end{thebibliography}
\end{document}